\documentclass[letterpaper,12pt]{article}

\usepackage{amsmath,amssymb,amsthm,bbm}
\usepackage{hyperref}
\hypersetup{colorlinks=false,pdfborder={0 0 0}}
\usepackage{color}
\usepackage{float}
\usepackage{graphicx}
\usepackage[small]{caption}
\usepackage[margin=1in]{geometry}
\usepackage{algorithm,algorithmic}
\usepackage{parskip}
\usepackage{young}
\usepackage{url}
\usepackage{bbm}

\renewcommand\P{\mathbb P}

\newcommand\cL{\mathcal L}

\newcommand\ignore[1]{ }

\newtheorem{theorem}{Theorem}[section]
\newtheorem{lemma}[theorem]{Lemma}

\newtheorem{proposition}[theorem]{Proposition}
\theoremstyle{definition}

\begingroup
    \makeatletter
    \@for\theoremstyle:=definition,remark,plain\do{%
        \expandafter\g@addto@macro\csname th@\theoremstyle\endcsname{%
            \addtolength\thm@preskip\parskip
            }%
        }
\endgroup

\begin{document}

\def\spacingset#1{\renewcommand{\baselinestretch}%
{#1}\small\normalsize} \spacingset{1}

\def\LS{LS}

\parskip 10pt

\begin{center}
\LARGE Exact sampling algorithms for Latin squares and Sudoku matrices via probabilistic divide-and-conquer\\[.3cm]
\normalsize Stephen DeSalvo\footnote{Department of Mathematics, University of California Los Angeles. stephendesalvo@math.ucla.edu}\\[.1cm]
\today\\[.5cm]
\parbox{.8\textwidth}{
\small
\textit{Abstract.}
We provide several algorithms for the exact, uniform random sampling of Latin squares and Sudoku matrices via probabilistic divide-and-conquer (PDC).  
Our approach divides the sample space into smaller pieces, samples each separately, and combines them in a manner which yields an exact sample from the target distribution. 
We demonstrate, in particular, a version of PDC in which one of the pieces is sampled using a brute force approach, which we dub \emph{almost deterministic second half}, as it is a generalization to a previous application of PDC for which one of the pieces is uniquely determined given the others. 
}
\end{center}
{

\noindent \textbf{MSC classes:} 60C05, 65C50, 60-04, 
}

\section{Introduction}

The random sampling of combinatorial structures is an active topic, with many available techniques. 
Several general methods have emerged which supply efficient algorithms in many cases of interest, e.g., rejection sampling~\cite{Rejection}, Boltzmann sampling~\cite{Boltzmann}, Markov chains~\cite{huber2015perfect, JerrumSinclair, MarkovChainBook, ProppWilson}.  
Each method presents difficulties: rejection sampling provides exact samples in finite time, but that finite time is often too large to be useful in practice; Boltzmann sampling produces an object which does not always satisfy the restrictions; rapidly mixing Markov chains are often easily fashioned to a problem, along with bounds on the mixing time (though not always), but they are notably inexact in finite time; 
and while Markov chain coupling from the past is an exact sampling method, it requires constructing a coupling over the complete set of objects, which can be practical when there is some type of monotonicity. 

There are many effective approaches to random sampling from the set of $9\times 9$ Sudoku matrices, which is defined as the collection of $9 \times 9$ integer-valued matrices such that each row and column is a permutation of the set $\{1,\ldots, 9\}$, and certain $3\times 3$ sub-blocks (see~\eqref{subblocks}) are also permutations of $\{1,\ldots, 9\}$. 
An efficient backtracking algorithm is utilized in \cite{DeSalvoSudoku}, though no quantitative bounds are given for the bias.  
Another approach is to fashion a Markov chain with uniform stationary distribution; this was described in~\cite{SudokuMC}, although it was not proved to be rapidly mixing.
Importance sampling delivers a collection of Sudoku matrices as well as weights which allows one to calculate unbiased estimates of statistical parameters; this was done in \cite{RidderIS}.  

We present several algorithms which are an application of probabilistic divide-and-conquer (PDC) \cite{PDC}. 
The central idea is the division of the sample space into two smaller parts, each of which can be sampled separately, and then appropriately pieced back together to form an exact sample. 
The method is versatile enough to be adapted to other types of numerical tables, and we have chosen the set of Latin squares as a natural complementary structure; a Latin square of order~$n$ is an $n \times n$ integer-valued matrix such that each row and column is a permutation of $\{1,\ldots, n\}$. 
There is also a generalization of Sudoku matrices which can be sampled using PDC, presented in Section~\ref{Sudoku:RC}. 

There are many ways to implement PDC, and in this paper we implement a version called \emph{almost deterministic second half}, which divides the sample space into two pieces, one of which can be sampled efficiently using a brute force approach.   There are other PDC algorithms: self-similar PDC (see~\cite[Section~3.5]{PDC}) often provides an asymptotically efficient algorithm, but requires detailed information about the sample space; PDC with deterministic second half (see~\cite[Section~3.3]{PDC},~\cite{PDCDSH}) requires almost no information about the sample space, and is such that given one of the pieces, the second piece is uniquely determined; PDC with the recursive method (see~\cite{DeSalvoImprovements}) is such that given one of the pieces, the second piece can be efficiently sampled using the recursive method~\cite{NW, NWBook}; the advantage of our current approach is that it can be customized to suit the available knowledge of the sample space.
When we do have detailed knowledge about the sample space, e.g., the number of ways of completing a partially generated object, as is the case for Sudoku matrices, we show how one can more optimally design a PDC algorithm.

The random sampling of Latin squares and Sudoku matrices is a notoriously difficult problem, and we can only claim in the present work to have made modest improvements to existing algorithms. 
PDC offers a versatile method for approaching the problem, and allows the fashioner to design an algorithm tailored to the particular computational aspects of the problem available. 
In addition, our algorithms do not require extensive auxiliary constructions or complicated transformations. 

The paper is organized as follows.
In Section~\ref{definitions} we give the definitions of a Latin Square of order $n$ and a Sudoku matrix.
In Section~\ref{PDC}, we introduce PDC and present some simple rejection sampling algorithms. 
In Section~\ref{section:Sudoku} we present a PDC algorithm for the random sampling of Sudoku matrices, an analysis of the cost, and some alternative PDC parameterizations. 
We do the same for Latin squares of order~$n$ in Section~\ref{section:Latin:squares}.
In Section~\ref{section:other} we review other approaches, and in Section~\ref{final:remarks} we explain our initial motivation and provide a link to a code implementation in C++. 

\section{Definitions}\label{definitions}

For any $n \geq 1$, a Latin square of order~$n$ is an $n \times n$ matrix which satisfies the following row and column constraints, heretofore referred to as the \emph{Latin square conditions}:
\begin{itemize}
\item each row, labelled $R_1, \ldots, R_n$ is a permutation of $\{1,2,\ldots, n\}$;
\item each column, labelled $C_1, \ldots, C_n$ is a permutation of $\{1,2,\ldots, n\}$.
\end{itemize}
We let $\LS_n$ denote the set of all Latin squares of order $n$.  

By a \emph{partially completed} matrix, e.g., partially completed Latin square, we mean a matrix that has some subset of entries filled in which do not a priori violate the conditions. 
No assumption is implied in general as to whether such a matrix can be completed. 
We shall primarily be interested in partially completed matrices whose first $k$ rows are filled in, although there are several possible applications of PDC which could be exploited in more general circumstances, see Section~\ref{other:PDC:algorithms:LS}. 
A partially completed Latin square of order~$n$ with the first $k$ rows filled in is called a \emph{$k \times n$ Latin rectangle}, for which much is already known, see for example~\cite{godsil1990asymptotic}. 

A Sudoku matrix is a Latin square of order $9$ which also satisfies the additional block constraints: 
there are nine $3 \times 3$ sub-blocks, labelled $B_1, B_2, \ldots, B_9$, each of which is a permutation of $\{1,2,\ldots, 9\}$; these blocks are indicated below.
\begin{equation}\label{subblocks}
\begin{array}{|ccc|ccc|ccc|}\hline
 &  &  &  &  &  &  &  &  \\
 & B_1 & & & B_2 & & &B_3 & \\
 &  &  &  &  &  &  &  &  \\ \hline
 &  &  &  &  &  &  &  &  \\
 & B_4 & & & B_5 & & &B_6 & \\
 &  &  &  &  &  &  &  &  \\ \hline
 &  &  &  &  &  &  &  &  \\
 & B_7 & & & B_8 & & &B_9 & \\
 &  &  &  &  &  &  &  &  \\ \hline
\end{array}
\end{equation}
We refer to these three constraints -- row, column and block -- as the \emph{Sudoku conditions}, and we let~$S$ denote the set of all Sudoku matrices.  

In the calculations that follow, $a \approx b$ is simply the elementary meaning where $a$ and/or $b$ have been rounded to some nearest value.  
For a finite set $A$, the notation $|A|$ will denote its cardinality, i.e., the number of elements in the set~$A$. 
In addition, we introduce notation for asymptotic analysis: suppose $f(n)$ and $g(n)$ are positive functions of a positive integer $n$, then
\begin{itemize}
\item we say $f(n) = O(g(n))$ if there exists a $C>0$ such that $\frac{f(n)}{g(n)} < C$ for all $n$ larger than some $n_0$;
\item we say $f(n) = o(g(n))$ if $\frac{f(n)}{g(n)} \to 0$ as $n$ tends to infinity;
\item we say $f(n) \sim g(n)$ if $\lim_{n\to\infty} \frac{f(n)}{g(n)} = 1$.  
\end{itemize}
For real-valued $x$, let 
\begin{itemize}
\item $[x]$ denote the integer closest to $x$; i.e., $x$ rounded to the nearest integer;
\item $\lfloor x \rfloor$ denote the largest integer smaller or equal to $x$, also known as the \emph{floor} of $x$.
\end{itemize}

\section{Probabilistic divide-and-conquer}
\label{PDC}

We begin by highlighting various simple algorithms for obtaining an element of $S$, a Sudoku matrix, uniformly at random (u.a.r.) over all possibilities. 
Each of the approaches can be adapted to sampling an element of $LS_n$ in a straightforward manner, and so we specialize to $S$ for concreteness. 

\begin{enumerate}
\item Sample each entry independent, and identically distributed (i.i.d.)~uniform over the set $\{1,2,\ldots,9\}$; restart if any of the Sudoku conditions are not satisfied.
\item Sample each row (or column or block) as i.i.d.~uniform over the set of all permutations of $\{1,2,\ldots, 9\}$; restart if any of the Sudoku conditions are not satisfied.
\item Let the first row be $(1,2,\ldots, 9)$.  Sample rows 2 through 8 as i.i.d.~uniform over the set of all fixed--point free permutations of $\{1,2,\ldots, 9\}$; if row 9 can be completed, with all Sudoku conditions satisfied, then complete it, and finally apply a random permutation of $\{1,\ldots, 9\}$ to all entries and return the matrix; otherwise, restart.
\end{enumerate}
These algorithms are generally known as \emph{rejection algorithms}, since the target set of objects $S$, also referred to as the target region, lies within the complete sample space $\Omega$.
Since the algorithm generates a sample uniformly over all elements of $\Omega$, any Sudoku matrix generated by such a hard rejection algorithm is also uniform over $S$ (see, e.g.,~\cite{Yordzhev2} for a more detailed analysis in a more general setting).

The expected number of times one must sample elements of $\Omega$ before obtaining a sample from $S$ is simply the quotient $|\Omega| / |S|$. 
It was shown in \cite{FelgenJarvis} that the number of Sudoku matrices is exactly 
\[ |S| = 6670903752021072936960 \approx 6.67 \times 10^{21}. 
\]
This number is too large to simply list all of the elements and select one at random.

There are $9^{81} \approx 2.0 \times 10^{77}$ matrices in $\Omega_1 := \{1,\ldots,9\}^{9\times 9}$, so for the first, most na\"ive approach of i.i.d.~entry sampling, the expected number of times one must sample from $\Omega_1$ is $|\Omega_1|/|S| \approx 3.3\times 10^{56}$.
There are $|\Omega_2| := \binom{9!}{9}\approx 3.0\times 10^{44}$ different ways of selecting $9$ distinct permutations of $\{1,2,\ldots, 9\}$, whence, $|\Omega_2|/|S| \approx 4.5\times 10^{22}$.  

The third algorithm is the most respectable of the three, as it eliminates two of the rows and reduces the number of possible permutations from $9!$ to $[9!/e]$. 
However, since we are fixing the first row, we can only sample from those elements of $S$ which also have the same first row; thus, we must normalize by $|S|/9!$ instead of $|S|$.
We have $\Omega_3 := \{$the set of all collections of seven fixed--point free permutations of $\{1,2,\ldots,9\}\}$, hence,
\[ \frac{|\Omega_3|}{|S|/9!} \approx \frac{ (9!/e)^7 }{1.8\times 10^{16}} \approx 4.2\times 10^{19}. \]
None of these rejection rates is practical.

At this stage, there are many different ways to proceed.
One approach is to relax the demand that the distribution on elements of $S$ is \emph{exactly} uniform, and instead adopt algorithms which are faster but not guaranteed to be uniform.
Another approach is to take advantage of more detailed properties of the elements in $S,$ which we now explore using PDC.

Rejection sampling algorithms can very often be improved, with very little extra information known about the target region, using PDC~\cite{PDC}.  
We assume a decomposition of the sample space $\Omega$ of the form $\Omega = \mathcal{A} \times \mathcal{B}$, where elements in $\Omega$ can be sampled using a distribution on $\mathcal{A}$, and a distribution on $\mathcal{B}$, such that random variables $A \in \mathcal{A}$ and $B \in \mathcal{B}$ are mutually independent; we denote their distributions by $\cL(A)$ and $\cL(B)$, respectively.  
In addition, we assume the target distribution on $S \subset \Omega$, denoted by $\cL(S)$, is of the form 
\[ \cL(S) = \cL( (A,B) | h(A,B)=1 ), 
\]
where $h: \mathcal{A}\times \mathcal{B} \to \{0,1\}$ is some (measurable) functional which determines whether or not the pair $(A,B)$ lies in the target set.
The PDC Lemma \cite[Lemma 2.1]{PDC} states that, assuming $\mathbb{P}(h(A,B)=1) > 0$, we have $\cL(S) = \cL(X,Y)$, where 
\[ \cL(X) = \cL( A | h(A,B) = 1), \qquad \cL(Y | X = x) = \cL(B | h(x,B) = 1). \]
An algorithm to sample from $\cL(S)$ is then as follows (see~\cite[Algorithm~2]{PDC}): 
\begin{enumerate}
\item Generate sample from $\cL(A\, |\, h(A,B) = 1),$ call it $x$.
\item Generate sample from $\cL(B\, |\, h(x,B) = 1),$ call it $y$.
\item Return $(x,y)$.
\end{enumerate} 

Often, however, the conditional distributions are not known, and so the more practical PDC algorithm utilizes von Neumann's rejection sampling approach \cite{Rejection}, which allows us to sample from the conditional distribution $\cL(A | h(A,B)=1)$ using $\cL(A)$ and a biased coin.  

\begin{algorithm}[H]{\rm
\begin{algorithmic}[1] 
\STATE Fix any $\alpha$ such that $\displaystyle \max_{a\in \mathcal{A}} \P(h(a,B)=0) \leq \alpha \leq 1$.  \\
\STATE For each $a \in \mathcal{A}$, define $\displaystyle t(a) := \frac{\P(h(a,B)=1)}{\alpha}$. \\
\STATE\label{line:A} Generate sample from $\cL(A),$ call it $a$.
\STATE Reject $a$ with probability $1-t(a)$; \ otherwise, restart from Line~\ref{line:A}.
\STATE Generate sample from $\cL(B\, |\, h(a,B) = 1),$ call it $y$.
\STATE Return $(a,y)$.
\end{algorithmic}}
\caption{\cite{PDC} Probabilistic Divide-and-Conquer via von Neumann} 
\label{PDC procedure von Neumann}
\end{algorithm}

Presently, in order to sample directly from the conditional distribution $\cL(B\, |\, h(a,B) = 1),$ we choose $\mathcal{A}$ such that $\cL(B\, |\, h(a,B) = 1)$ can be sampled effectively using a brute force approach; we call this \emph{PDC with almost deterministic second half}.  
Also, since the rejection probability $t(a)$ must satisfy, for some $0<\alpha\leq 1$, 
\[ t(a) =  \frac{\P( h(a,B) = 1)}{\alpha} \leq 1, \qquad \text{for all $a\in \mathcal{A}$,}\]
the optimal choice of $\alpha$ is $\max_a \P(h(a,B)=1)$, \emph{but the algorithm is still valid and unbiased for any $\alpha$ larger than this value.}
Thus, there are many ways to apply PDC, starting with an appropriate selection of sets $\mathcal{A}$ and $\mathcal{B}$, and it suffices to choose any universal upper bound of $\P(h(a,B)=1)$, where greater efficiency is achieved by selecting the optimal value of $\alpha$.  

\section{Sudoku matrices}
\label{section:Sudoku}

\subsection{A PDC algorithm}
\label{section:Sudoku:algorithm}
We now describe an algorithm for Sudoku matrices which takes advantage of Algorithm~\ref{PDC procedure von Neumann}. 
Given the first three rows $R_1, R_2, R_3$ of a partially completed Sudoku matrix, let $P_9' \equiv P_9'(R_1, R_2, R_3) \equiv P_9'(B_1, B_2, B_3)$ denote the set of all possible permutations of $\{1,2,\ldots, 9\}$ which would not violate the Sudoku conditions if placed in row four. 
Also, we denote by $U$ a random variable with the uniform distribution over the interval $(0,1)$, independent of all other random variables, and each occurrence of $u$ in an algorithm means to generate a random variate from the distribution of $U$. 

\begin{algorithm}[H]
\caption{PDC Sudoku algorithm.}
\begin{algorithmic}[1]
\STATE Let $B_1 = (1,2,\ldots,9)$.
\STATE Select $B_2, B_3$ in proportion to the number of completable Sudoku matrices.\label{Sudoku:special}
\STATE Generate $(R_4, R_5, R_6, R_7),$ each an i.i.d.~uniformly random element from $P_9'$.  \label{line:generate}
\STATE Let $d$ denote the number of possible completions given $(R_1, R_2, R_3, R_4, R_5, R_6, R_7)$.
\IF {$u < \frac{d}{16}$,}\label{line:rejection}
 \STATE\label{completion} Select a completion uniformly at random from the $d$ possible completions.
\ELSE \STATE Goto 3.
\ENDIF
\STATE Apply a random permutation to $(C_4, C_5, C_6)$.
\STATE Apply a random permutation to $(C_7, C_8, C_9)$.
\STATE Apply a random permutation to $( (C_4, C_5, C_6), (C_7, C_8, C_9) )$.
\STATE Apply a random permutation of $\{1,2,\ldots, 9\}$ to the entries and return.
\end{algorithmic}
\label{Sudoku}
\end{algorithm}

Line~2 demands some further explanation.
By considering various symmetries in the permutations of columns, it was shown in \cite{FelgenJarvis} that there are only 36288 essentially unique completions to blocks $B_2$ and $B_3$;\footnote{This number can be simplified to 71 after taking more symmetries into account~\cite{SudokuWebsite}.} and, for each configuration of blocks $B_1$, $B_2$, $B_3$, there is a certain number of completable Sudoku matrices, which was calculated via brute force.  
A downloadable file containing these enumerations is available online, see~\cite{SudokuWebsite}; 
the first few lines look like the following:
\begin{verbatim}
[456789,789123,123456] => 108374976
[456789,789123,123465] => 102543168
[456789,789123,123546] => 102543168
[456789,789123,123564] => 100231616
[456789,789123,123645] => 100231616
...
\end{verbatim}
The list of three 6--tuples represent columns 3 through 9 of the first three rows, and the integer value at the end of the arrow counts the number of completable Sudoku matrices given these first three rows.  For example, the text file says that the matrix with top three rows given by
\[
\begin{Young}
1 & 2 &3 & 4& 5& 6& 7& 8&9 \cr
4 & 5 &6 & 7& $8$&9 &1 &2 &3 \cr
7 & 8 & 9 & 1 & 2 & 3 & 6 & 4 & 5 \cr
 \end{Young}
\]
has precisely 100231616 completions of the bottom 6 rows which are valid Sudoku matrices.
Line~\ref{Sudoku:special} is thus a straightforward sampling of the entries in $B_2$ and $B_3,$ \emph{in proportion to the number of completable Sudoku matrices}.
There are two reasonable ways to perform this sampling.
Denote by $x_i$ the number of possible completions given outcome $i$, $i=1,\ldots, 36288$.  
One can normalize by the sum of all the completions and generate a value using the discrete probability distribution $\{x_i / \sum_{j} x_j\}_{i=1,\ldots,36288}$; this is not so ideal in general due to floating point considerations.
Instead, we apply PDC in our implementation by sampling uniformly from among the 36288 possibilities first, say we select outcome $i$, and we accept this sample with probability $x_i / \max_j x_j$.  

Line~\ref{completion} is then the last task of the algorithm, and for this we have implemented a brute force approach which enumerates through all possibilities and picks one uniformly at random. 
Specifically, each column has a set of two elements $\{a_i, b_i\}$, $i=1,\ldots, 9$, which must be used in that column in the final two rows, in some order, and we are able to greatly reduce the enumeration of all a priori $2^9 = 512$ possibilities by taking various symmetries into account; see for example the proof of Lemma~\ref{upper bound}.  
 
\subsection{Cost of Algorithm~\ref{Sudoku}}

\begin{lemma}\label{upper bound}
Assume the first seven rows of a Sudoku matrix have been filled in, and none of the Sudoku conditions are violated based on these first seven rows.  Then the total number of possible completions of this matrix which also do not violate the Sudoku conditions is at most 16. This bound cannot be improved.
\end{lemma}
\begin{proof}
WLOG, assume the last two rows of column $j$ can be completed by $\{2j-1,2j\}$ for $j=1,2,3$.  We shall denote this by the following table
\[
\begin{Young}
1 & 3 &5 & & & & & & \cr
2 & 4 &6 & & & & & & \cr
 \end{Young}
\]
By interchanging $2j-1$ and $2j$ for $j=1,2,3$, we obtain $2^3$ possible combinations. 
For the elements $7,8,9$, there are two possibilities: either $7$ and $8$ share the same column, or they do not.  
\begin{itemize}
\item In the case where they share the same column, we have WLOG
\[
\begin{Young}
1 & 3 &5 & 7& 9& & & & \cr
2 & 4 &6 & 8& $a$& & & & \cr
 \end{Young}
\]
which implies another factor of 2 by exchanging the role of $7$ and $8$; however, whichever element is paired with the 9, say $a=5$, determines which row the 9 lies in, and so there are at most 16 possible completions. 
\item In the case where $7$ and $8$ do not share the same column, whichever element the $7, 8,$ and $9$ are paired with uniquely determine their row, and so there are at most $8$ possible completions in this case. 
 \end{itemize}
 
Thus, there are at most $2^4=16$ possible completions given the first seven rows of a partially completed Sudoku matrix.

The following matrix (in reduced form) was generated by Algorithm~\ref{Sudoku}.
The first seven rows of this matrix allow for 16 potential completions, thus, 16 is a tight upper bound.
\[\begin{Young}
1 & 2 & 3 & 4 & 5 & 9 & 6 & 7 & 8 \cr
4 & 5 & 6 & 7 & 8 & 2 & 3 & 1 & 9 \cr
7 & 8 & 9 & 1 & 6 & 3 & 2 & 5 & 4 \cr
9 & 4 & 8 & 2 & 1 & 7 & 5 & 6 & 3 \cr
5 & 7 & 2 & 3 & 4 & 6 & 8 & 9 & 1 \cr
6 & 3 & 1 & 5 & 9 & 8 & 4 & 2 & 7 \cr
3 & 6 & 7 & 9 & 2 & 4 & 1 & 8 & 5 \cr
8 & 9 & 5 & 6 & 3 & 1 & 7 & 4 & 2 \cr
2 & 1 & 4 & 8 & 7 & 5 & 9 & 3 & 6 \cr
\end{Young}
\qedhere\]
\end{proof}

We now justify Line~2 of Algorithm~\ref{Sudoku}.
In \cite{FelgenJarvis}, the total number of Sudoku matrices is found by first reducing by symmetry the total number of possible completed first three rows; the number given is 36288.
For each of these 36288 configurations of the first three rows, a brute force enumeration is performed for the number of possible Sudoku matrices that could be completed. 
\begin{lemma}[\cite{FelgenJarvis}]\label{table}
The table of 36288 possible configurations of $B_2$ and $B_3$ contains, for each configuration, the number of possible Sudoku matrices that can be completed given the first three rows.  These numbers all lie between 94888576 and 108374976.
\end{lemma}
The fact that these values are relatively constant has a number of quantitative and qualitative interpretations.
From a rejection sampling perspective, it means that we could sample a configuration uniformly at random and then reject each generated configuration in proportion to the number of completions, \emph{normalized by the maximum possible}, i.e., 108374976.
Thus, the probability of rejecting a sample is at worst $1 - 94888576 / 108374976 \approx 0.12444$.  
Practically, this means that sampling of the first three rows uniformly from among the 36288 possible choices and applying a rejection is efficient. 
Also, it says that each configuration is approximately as completable as any other configuration; i.e., accepting any completable configuration of the first three rows introduces a small bias.

Thus, the main cost of the algorithm is the rejection sampling of rows 4 through 7.  
Through a complete enumeration of all 36288 cases, the number of permutations in $P_9'$ is always between 12000 and 12096.  
There are thus at most $12096^4 \approx 2\times 10^{16}$ different combinations of permutations that can be placed in rows 4 through 7, but only about $10^8$ valid completions to a Sudoku matrix are possible.
This means that on average we must sample about $2\times 10^8$ 4--tuples before we get lucky and obtain a quadruplet that yields a completable Sudoku matrix.
Even after we obtain a completable quadruplet, however, we must still survive the final rejection in Line~\ref{line:rejection}, which comes from the fact that different quadruplets yield a different number of completable Sudoku matrices.

\begin{lemma}[\cite{devroye}]\label{rejection lemma}
In a rejection sampling algorithm, let $f$ denote the target distribution, and $g$ denote the sampling distribution; 
then the rejection probability, see for example Line~\ref{line:A} in Algorithm~\ref{PDC procedure von Neumann}, for exact sampling from distribution $f$ is given by 
\[ t(a) = C\, \frac{f(a)}{g(a)}, \qquad a \in \mathcal{A},\]
where the optimal value for $C$ is given by 
\[ C^\ast := \left(\max_a \frac{f(a)}{g(a)}\right)^{-1}. \]
The algorithm is unbiased for any $C \leq C^\ast$. 
Algorithms which utilize rejection sampling have a geometrically distributed number of generations before acceptance, with expected value given by $\frac{1}{C}$.  
\end{lemma}

\begin{theorem}\label{Sudoku theorem}
Algorithm~\ref{Sudoku} samples uniformly over the set of all Sudoku matrices.
The number of times to sample blocks $B_2$ and $B_3$ by choosing one of the 36288 possible choices u.a.r.~and using rejection sampling is geometrically distributed with expected value 
\begin{equation}\label{rejection 1} \frac{36288\cdot 108374976}{3546146300288} \approx 1.1. \end{equation}
The number of times to sample rows $R_4, R_5, R_6, R_7$ is geometrically distributed with expected value at least $3.0 \times 10^9$ and at most $3.6 \times 10^9$. 
\end{theorem}
\begin{proof}
The uniformity of the algorithm follows by two applications of probabilistic divide-and-conquer given by Algorithm~\ref{PDC procedure von Neumann}.  We now calculate the precise rejection probabilities.

The first application of PDC is given by the division 
\[ \text{$A = (B_1, B_2, B_3)$, \quad $B = (R_4, R_5, R_6, R_7, R_8, R_9)$.} \]
The rejection probabilities $t(a)$ are given by the number of possible completions normalized by the maximum number of possible completions; these values are from Lemma~\ref{table}.
We index each of the possible configurations of $B_2$ and $B_3$ by $i=1,2,\ldots, 36288$.  
Then we select one uniformly at random, hence $g(i) = 1/36288$, but the distribution desired is the one that weights each configuration in proportion to the number of possible completions, denoted by $x_i$ previously, which is contained in the table of values computed in \cite{FelgenJarvis}.  That is, 
\[ f(i) = \frac{ x_i }{3546146300288}, \qquad i=1,\ldots, 36288, \]
where $3546146300288 = \sum_j x_j$ is the total number of possible completions by all elements in the table, which makes $f(i)$ a probability distribution.
The expected number of times to sample from blocks $B_2$ and $B_3$ is thus 
\begin{equation*} \frac{1}{C} = \left(\max_i \frac{f(i)}{g(i)}\right) = \frac{36288\cdot (\max_i x_i)}{3546146300288}, \end{equation*}
where 
\[ \max_i x_i = 108374976 \]
was computed by brute force over the set of all possibilities. 
We have just proved~\eqref{rejection 1}. 

The second application of PDC is the division 
\[ \text{$A = (R_4, R_5, R_6, R_7)$, \quad $B = (R_8, R_9)$.}\]
  The rejection probability $t(a)$ is given by the number of possible completions (calculated via brute force) given $a = (r_4, \ldots, r_7)$, normalized by 16, the optimal upper bound provided by Lemma~\ref{upper bound}.
Rather than sample from the set of all permutations of $\{1,\ldots,9\}$, since we have already accepted the first three rows of the matrix, we can automatically discard any permutations which would violate the Sudoku conditions.  Thus, we sample these four rows from the set $P_9'$, which depends on the particular configuration of $B_2$ and $B_3$ accepted in the first part.  
By brute force calculation over all 36288 possible first three rows, we have found that $12000 \leq |P_9'| \leq 12096.$
Thus,
\[ \frac{1}{12096^4} \leq g(j) = \frac{1}{|P_9'|^4} \leq \frac{1}{12000^4}, \qquad \text{for all $j$}. \]
Similarly, we have
\[ \frac{8}{108374976} \leq f(j) = \frac{\text{\# completions given first seven rows}}{\text{\# completions given first three rows}} \leq \frac{16}{94888576}. \]
Whence,
\[ 3.0 \times 10^9 \approx  \frac{12000^4\times 16}{108374976} \leq \frac 1C = \left( \max_j \frac{f(j)}{g(j)} \right) \leq \frac{12096^4\times 16}{94888576} \approx 3.6 \times 10^9. \qedhere\]
\end{proof}

\subsection{Alternative PDC parameterizations for Sudoku matrices}
\label{other:PDC:algorithms}

There are two simple alternative approaches to Algorithm~\ref{Sudoku} that require minimal modification of the original algorithm; we indicate the replacement for lines~\ref{line:generate}--\ref{line:rejection} below in each case.

The first approach is to randomly sample $R_4$ through $R_8$.

3: Generate $(R_4, R_5, R_6, R_7, R_8),$ each an i.i.d.~uniformly random element of $P_9'$. \\
4: Let $d$ denote the number of possible completions given $(R_1, R_2, R_3, R_4, R_5, R_6, R_7, R_8)$. \\
5: {\bf if} {$d=1$,} {\bf then} 

That is, conditional on the existence of a completion to $R_9$, we may simply accept the sample $R_4$ through $R_8$ and fill in the \emph{unique} completion to $R_9$; this variation is called PDC deterministic second half \cite{PDC}, see also \cite{PDCDSH}. 
(In general, PDC deterministic second half also requires a rejection step, but in the case when $\cL(B\,|\,h(a,B)=1)$ is a uniform distribution over a singleton set for each $a \in \mathcal{A}$ which can be used to complete a sample, all samples are rejected with probability 0.)
This approach increases the probability of rejection in Line~\ref{line:rejection} considerably, see Proposition~\ref{prop:R8} below, but eliminates the task in Line~\ref{completion}.  

\begin{proposition}\label{prop:R8}
Suppose Line~\ref{line:generate} in Algorithm~\ref{Sudoku} instead samples rows~$R_4$ through $R_8$.  
Then the expected number of rejections until the first eight rows are a partially completed Sudoku matrix is within the interval
\[ \left[\frac{12000^5}{108374976},  \frac{12096^5}{94888576} \right] \approx [2.3, 2.73] \times 10^{12}. \]
Then, once these first eight rows are a partially completed Sudoku matrix, the probability of rejection is 0, and the last row is uniquely determined. 
\end{proposition}

The proof of Proposition~\ref{prop:R8} is straightforward, with upper bounds on the number of possible completions provided by Lemma~\ref{completion:lemma}, and the rejection probability provided by~\cite[Theorem~7.1]{PDCDSH}. 

Alternatively, we can instead randomly sample $R_4$ through $R_6$, and reject these three rows in proportion to the number of possible completions of the last three rows, indicated below: 

3: Generate $(R_4, R_5, R_6),$ each an i.i.d.~uniformly random element from $P_9'$. \\
4: Let $d$ denote the number of possible completions given $(R_1, R_2, R_3, R_4, R_5, R_6)$. \\
5: {\bf if} {$u < \frac{d}{\lfloor 2^{4.5}6^3\rfloor}$,} {\bf then} 

From a theoretical point of view, this approach is more ideal, since the probability of rejection in Line~\ref{line:rejection} is smaller (since $d$ is much less likely to be 0). 
However, completing the sample as in Line~\ref{completion} can also be a bottleneck, since now there are a priori $3^9 \approx 2\times 10^4$ possible completions that must be checked, even though we can greatly reduce this number by considering various symmetries. 

\begin{proposition}\label{prop:R6}
Suppose Line~\ref{line:generate} in Algorithm~\ref{Sudoku} instead samples rows~$R_4$ through $R_6$.  
Then the expected number of rejections is within the interval 
\[ \left[\frac{\lfloor 2^{4.5}\,6^{3}\rfloor \, 12000^3}{108374976},  \frac{\lfloor 2^{4.5}\,6^3\rfloor \, 12096^3}{94888576} \right] \approx [7.8, 9.1] \times 10^{7}. \]
Then, once these first six rows are determined, the last three rows can be completed in at most $\lfloor 2^{4.5} 6^3\rfloor$ ways. 
\end{proposition}

The upper bound on the number of possible completions in Proposition~\ref{prop:R6} is borrowed from the theory of Latin squares, see Section~\ref{sect:Latin:analysis}, specifically Lemma~\ref{completion:lemma} using the values $n=9$ and $k=6$. 

Despite the much smaller expected number of rejections in Proposition~\ref{prop:R6}, the reason why we champion Algorithm~\ref{Sudoku} presently is due to our brute force implementation of sampling from $\cL(B\, |\, h(a,B)=1)$, which is optimized to enumerate between all potential completions using fast bit--wise operations in C++, whereas in the more general case we do not have this option encoded.  
We emphasize that while we have not yet efficiently coded this algorithm for more general cases, should such a module be coded efficiently it could very easily be faster in practice than Algorithm~\ref{Sudoku}.

In addition, it may be possible to improve upon the upper bound of $\lfloor 2^{4.5} 6^3\rfloor$ in this special case. 
To help facilitate such an endeavor, we give a motivating example. 
The upper bound in Lemma~\ref{completion:lemma} for the number of completable $(n-3) \times n$ Latin rectangles is $2^{n/2}\,6^{n/3}$.  
With $n=9$, we have $\lfloor 2^{n/2}\,6^{n/3}\rfloor = 4887.$  However, the largest observed value for partially completed Sudoku matrices in a sample of size 1000 was 288.  
If, in fact, 288 is the smallest upper bound, and was used instead of 4887, then the run--time of the algorithm which samples all but the final three rows would be reduced by a factor of about 17.  

\subsection{Further reduction by symmetries}
When various symmetries are taken into account, one can reduce the total number of Sudoku matrices to $5472730538 \approx 5.4\times 10^9$ essentially different Sudoku matrices, see~\cite{RusselJarvis}.
This number is certainly more practical, and a comprehensive list of such matrices could be stored for random access if the memory was available, thus offering an algorithm which is $O(1)$ in terms of our costing of random generation; of course, one would also have to implement the various transformations, but this cost is certainly not prohibitive. 
The advantage with our approach is that its memory requirements are entirely practical, even for a computer of modest means, and can be generalized in a straightforward manner.
Nevertheless, should Algorithm~\ref{Sudoku} be deemed not efficient enough, a more efficient random sampling algorithm may be achievable by taking these symmetries into account.

\section{Latin squares} 
 \label{section:Latin:squares}
\subsection{A PDC algorithm} 

The upper bound in Lemma~\ref{upper bound} is specialized to the set of Sudoku matrices, and in general such tight bounds are not available.
Nevertheless, we present a similar PDC algorithm in this section for Latin squares of order~$n$ for any $n \geq 5$. 

\begin{algorithm}[H]
\caption{PDC Uniform sampling from $LS_n$}
\begin{algorithmic}[1]
\STATE Let $R_1 = (1,\ldots,n)$.
\FOR {$i=2,\ldots, n-3$}  
    \STATE Generate $R_i$ uniformly from the set of fixed--point free permutations of $\{1,\ldots, n\}$. 
\ENDFOR
\STATE Let $d$ denote the number of possible completions given $(R_1, \ldots, R_{n-3})$.\label{line:d:LS}
\IF {$U < \frac{d}{\lfloor 6^{\frac{n}{3}}2^{ \frac{n}{2}}\rfloor}$}\label{line:rejection:LS}
\STATE Select a completion uniformly at random from the $d$ possible completions.\label{line:complete}
\ELSE
\STATE Goto 2.
\ENDIF
\STATE Apply a random permutation to the rows of the matrix.
\STATE Apply a random permutation of $\{1,\ldots,n\}$ to the entries and return.
\end{algorithmic}
\label{Sudoku:n}
\end{algorithm}

As in Algorithm~\ref{Sudoku}, the selection of a completion in Line~\ref{line:complete} once the first set of rows has been accepted can be performed using brute force enumeration, or any alternative method which uniformly samples from the set of possible completions.

\subsection{Cost of Algorithm~\ref{Sudoku:n}}
\label{sect:Latin:analysis}
We now discuss the parameterized algorithm for Latin squares of order $n$. 
There are two competing styles of analysis: one champions an analysis for a particular, finite value of $n$ chosen in advance, e.g., $n=9$ as in the previous section; the other is an asymptotic analysis of how the algorithm scales as $n$ tends to infinity. 
Whereas our previous analysis for Sudoku matrices was specific to a fixed size~$9$, we shall now be primarily interested in how Algorithm~\ref{Sudoku:n} scales as $n$ tends to infinity. 

For a partially completed Latin square of order $n$ with the first $k$ rows filled in, a more general bound for all $n$ and $k$ can be obtained using upper bounds on the permanent of an $n\times n$ matrix taking values in the set $\{0,1\}$, as was done in~\cite{vanLintWilson}. 
We now adapt the argument in our setting. 

\begin{lemma}\label{completion:lemma}
Suppose the first $k$ rows of any Latin square of order $n$ have been filled in, $k < n$, without violating any of the Latin square conditions. 
\begin{enumerate}
\item There always exists at least one completion to a full Latin square of order~$n$, and when $k=n-1$ there is exactly one completion. 
\item Let $C_{n,k}$ denote the maximum possible number of completions to a full Latin square of order~$n$ given these first $k$ rows.  Then we have 
\begin{equation}\label{Cnk} C_{n,k} \leq  \prod_{\ell=1}^{n-k} (\ell!)^{n/\ell}. \end{equation}
\end{enumerate}
\end{lemma}
\begin{proof}
First, the existence of at least one completion of a partially completed Latin square is given by~\cite[Theorem~17.1]{vanLintWilson}, and the uniqueness of the completion when $k=n-1$ follows by taking $k=n-1$ in inequality~\eqref{Cnk}. 

The remaining argument is the same as~\cite[Proof of Theorem~17.3]{vanLintWilson}.  
That is, we let $B$ denote the $n\times n$ matrix where the $(i,j)$th entry is $1$ if element $i$ does not appear in the first $k$ rows of column $j$, and 0 otherwise. 
Thus, the row sums of $B$ are all $n-k$, and so by~\cite[Theorem~11.5]{vanLintWilson}, we have 
\[ per\, B \leq (n-k)!^{n/(n-k)}, \]
where $per\, B$ denotes the permanent of the matrix $B$, which is also the number of ways of completing the next row. 
Thus, by applying the above inequality $n-k-1$ times, for completing rows $k+1, k+2, \ldots n-1$, we obtain
\[ C_{n,k} \leq \prod_{\ell=k}^{n-1} (n-\ell)!^{n/(n-\ell)}. \qedhere\]
\end{proof}

We shall also need for our analysis of asymptotic runtimes to estimate the magnitude of $|LS_n|$, which is still not precisely known asymptotically as $n$ tends to infinity, although the following asymptotic result is contained in~\cite[Theorem~17.3]{vanLintWilson}.  

\begin{theorem}[\cite{vanLintWilson}]\label{Ln:asymptotic}
We have 
\[ |LS_n|^{1/n^2} \sim e^{-2}\, n. \]
\end{theorem}

There are, however, well-known upper and lower bounds which shall at least give us the ability to estimate the efficiency of our algorithms, which we give below.  See~\cite[Theorem~17.2]{vanLintWilson} for the lower bound and see~\cite[the proof of Theorem~17.3]{vanLintWilson} for the upper bound.  

\begin{theorem}[\cite{vanLintWilson}]\label{LS:bounds}
We have 
\[ \frac{(n!)^{2n}}{n^{n^2}} \leq |LS_n| \leq \prod_{k=1}^n (k!)^{n/k}. \]
\end{theorem}

\begin{theorem}\label{Latin square theorem}
Algorithm~\ref{Sudoku:n} samples uniformly over the set of all Latin squares of order $n$.
The number of times to sample rows $R_2,\ldots, R_{n-3}$ is geometrically distributed with expected value 
\[E_n := \frac{\lfloor 6^{n/3}2^{n/2}\rfloor \left([n!/e]\right)^{n-4}}{|LS_n|/n!}, \qquad n \geq 5.\]
\end{theorem}
\begin{proof}
We note that $|LS_n|$ is normalized by $n!$ since we are taking the first row to be the permutation $(1,\ldots, n)$.  
Each of the rows $R_i$, $i=2,\ldots, n-3$, is sampled from the set of all fixed-point-free permutations, of which there are $[n!/e]$ such permutations, and our sample space is the set of all $(n-4)$--tuples of fixed-point-free permutations.  
The proof then proceeds in the same fashion as the proof of Theorem~\ref{Sudoku theorem}. 
\end{proof}

A table of values is below, calculated from the exact quantities and rounded to a few significant digits; the values for $|LS_n|$  can be found in the recent survey~\cite{Stones}.
In addition, we calculate the asymptotic rate of increase of $E_n^{1/n^2}$ in Proposition~\ref{E:growth}. 
\[
\begin{array}{|c|c|c|}\hline
n & |LS_n| & E_n \\ \hline\hline
 5 & 1.61\times 10^5 & 8.15 \\
 6 & 8.13\times 10^8 & 1.32 \times 10^2 \\
 7 & 6.15\times 10^{13} & 1.10 \times 10^4 \\
 8 & 1.09\times 10^{20} & 4.26\times 10^6 \\
 9 & 5.52\times 10^{27} & 8.41\times 10^9 \\
 10 & 9.98\times 10^{36} & 8.79\times 10^{13} \\
 11 & 7.77\times 10^{47} & 5.03\times 10^{18} \\ \hline
\end{array}
\]

\begin{proposition}\label{E:growth}
We have 
\[ E_n^{1/n^2} \sim e. \]
\end{proposition}
\begin{proof}
Recall Stirling's formula, which states that as $n$ tends to infinity, we have 
\[ n! \sim \frac{n^n}{e^n} \sqrt{2\pi\, n}\, , \]
so 
\[ n!^{1/n} \sim \frac{n}{e}\, . \]
We then have
\[ E_n^{1/n^2}  = \left(\frac{\lfloor 6^{n/3}2^{n/2}\rfloor \left([n!/e]\right)^{n-3}}{|LS_n|}\right)^{1/n^2}  \sim \frac{6^{1/(3n)}2^{1/(2n)} \left(n!/e\right)^{1/n-3/n^2}}{|LS_n|^{1/n^2}}  \sim \frac{n/e}{n/e^2}  \sim e.
\]
\end{proof}

These values indicate that the rejection probability is the dominating aspect of the computation. 
 See Section~\ref{other:PDC:algorithms:LS} for other ways to improve upon this rejection cost.

\subsection{Alternative PDC parameterizations for Latin squares}
\label{other:PDC:algorithms:LS}

Let $U_{n,k} := \prod_{\ell=1}^{n-k} (\ell!)^{n/\ell}$ denote the right-hand side of inequality~\eqref{Cnk}, i.e., an upper bound for the number of completions of a $k \times n$ Latin rectangle. 

Algorithm~\ref{Sudoku:n} can be generalized with a parameter $k \geq 2$  so that rows $R_2, \ldots, R_{k}$ are sampled in Line~\ref{line:d:LS} for the first step in PDC.  
In this case, the rejection probability in Line~\ref{line:rejection:LS} may be replaced with $d / U_{n,k}$, and the number of times to resample has expected value
\[  E_{n,k} := \frac{U_{n,k}\, [n!/e]^{n-k-1}}{L_n/n!}, \qquad n \geq 5, 1 \leq k \leq n-4.\]
We estimate the log of this rejection cost asymptotically as $n$ tends to infinity.  

\begin{proposition}\label{asymptotic:Enk}
For 
\begin{itemize}
\item $k = o(n)$, we have 
\[ \log E_{n,k} \sim n^2 \log(n);\]
\item $k \sim (1-t)\, n$, for any $0<t<1$, we have 
\[ \log E_{n,k} \sim (2t-1) n^2 \log(n);\]
\item $k = n - r\, n^\alpha$, for any $0<r<1$, and $0<\alpha<1$, we have 
\[ \log E_{n,k} \sim r(1+\alpha)n^{1+\alpha} \log(n) - n^2 \log(n). \]
\end{itemize}
\end{proposition}
\begin{proof}
To estimate $\log E_{n,k}$, we note that the proof of~\cite[Theorem~17.3]{vanLintWilson} contains an approach to estimate both $\log |LS_n|$ and $\log|U_{n,k}|$; that is, we have  
\[ \log |LS_n| \sim n^2 \log(n), \]
and in a similar fashion we obtain 
\[ \log U_{n,k} \sim n(n-k) \log(n-k). \]

Thus, asymptotically as $n$ and $n-k$ tend to infinity, we have 
\begin{align*} 
\log E_{n,k}  & = \log U_{n,k} + (n-k)\log(n!) - (n-k-1) - \log |LS_n| \\
    & \sim n\, (n-k) \log(n-k) + (n-k) (n \log(n)) - n^2 \log(n).
\end{align*}
The proposition now follows by a straightforward calculation. 
\end{proof}

There are several noteworthy aspects of the asymptotic analysis in Proposition~\ref{asymptotic:Enk}. 
First, when $k = o(n)$, there is asymptotically as much uncertainty in the first $k$ rows as there are in the entire ensemble of Latin squares of order $n$. 
Next, taking $k = n/2$, i.e., $t=2$, is a natural cutoff for when the uncertainty is no longer on the same exponential order as the entire ensemble, and indeed it is plausible that a self-similar PDC algorithm using this division may achieve an asymptotically optimal rejection rate. 
Finally, the final item in Proposition~\ref{asymptotic:Enk} is an example of a low-rank parameterization (see for example~\cite{PIPARCS}), and yields the greatest immediate potential for generalizing the presented algorithms in a way which is still analytically tractable. 

Of course, even with these divisions, one still has to sample from the remaining set of rows $R_{k+1}, \ldots, R_n$, which for $k$ moderately large is still a daunting task. 

One can also apply PDC more generally, not just to $k \times n$ Latin rectangles, by filling in any amount of partial information. 
There are many enticing partial results in this area, see~\cite[Theorem~17.4]{vanLintWilson} for an example which fills in an upper left rectangle, and~\cite[Theorem~17.5]{vanLintWilson} for an example which fills in an upper triangular section of entries. 
There are also many cautionary theoretical results as well, see~\cite[Figure~17.4]{vanLintWilson} and~\cite{colbourn1984complexity}. 

As a final note for this section, it is typical to reduce a Latin square by assuming the first row \emph{and} the first column are the identity permutation.  The number of such reduced Latin squares is then $|LS_n|/(n! (n-1)!)$, and it is often easier to work with this reduced set. 
This reduction could also be exploited for PDC. 
It is valid to replace Lines~1--3 in Algorithm~\ref{Sudoku:n} with the following:

1: Let $R_1 = (1,\ldots,n)$. \\
2: {\bf for} $i=2,\ldots, n-3$  \\
3: Generate $R_i$ uniformly from the set of fixed--point free permutations starting with $i$. \\

This restriction forces the first $n-3$ entries of column 1 to be the identity permutation $(1,2,\ldots,n-3)$ (and one could easily force the remaining entries in the column to be $(n-2, n-1, n)$ ). 
However, for $2 \leq i \leq n-3$, the set of all fixed--point free permutations of $n$ starting with an $i$ is given by the OEIS sequence $A000255$~\cite{OEIS}. 
In~\cite[Page~373]{Flajolet}, it is shown that the cardinality of this set is asymptotically $n!/e$, which is asymptotically the same as the number of fixed--point free permutations of $n$ that we have used in our current analysis.

\section{Other sampling algorithms}\label{sect alt}
\label{section:other}

\subsection{The $R \times C$ generalization of a Sudoku matrix}
\label{Sudoku:RC}
There is a generalization to a Sudoku matrix where the block constraint is modified to be a block with dimensions $R \times C$. 
The set of $(R,C)$-Sudoku matrices is then defined as the set of $R\,C \times R\, C$ matrices tiled with the $R \times C$ blocks, where each row, column, and $R\times C$ block is a permutation of $\{1, 2, \ldots, R\,C\}$. 
Sudoku matrices are the special case with $R=C=3$. 

In fact, Algorithm~\ref{Sudoku:n} is worded so that it can be applied to $(R,C)$-Sudoku matrices \emph{varbatim} using $n=R\,C$: one simply interprets Line~\ref{line:d:LS} and Line~\ref{line:complete} analogously. 
Even the denominator of the rejection probability, i.e., $\lfloor 6^{n/3} 2^{n/2}\rfloor$, can be used as an upper bound on the number of completable $(R,C)$-Sudoku matrices. 
Of course, there is some efficiency lost in the algorithm by using this generic upper bound, however, as was pointed out earlier, it does not affect the unbiased nature of the algorithm, only the efficiency, and should more efficient bounds become available, one could (and should!) substitute in those bounds instead.

\subsection{$S$-permutation matrices}
\label{S:permutation}
A recent approach to random sampling of Latin squares and Sudoku matrices is via permutation matrices, see~\cite{Dahl}; see also \cite{FontanaFractions, Fontana, YordzhevNumber}.  
A permutation matrix of order~$N$ is an $N \times N$ matrix with a exactly one 1 appearing in each row and column, and the rest of the entries are 0. 
An \emph{$S$-permutation matrix of order $N = n^2$} is an $N \times N$ permutation matrix with the additional constraint that only one~$1$ appears in each of the $n \times n$ sub-blocks indicated by the generalization in Section~\ref{Sudoku:RC} for $(n,n)$-Sudoku matrices. 
Central to this idea is the fact that each $(n,n)$-Sudoku matrices $A$ can be written uniquely as 
\begin{equation}\label{Sudoku:decomposition} A = 1\cdot P_1 + 2\cdot P_2 + \cdots + N\cdot P_N, \end{equation}
where $P_1, \ldots, P_N$ are each permutation matrices of order~$N$, with the additional property of being \emph{disjoint}, i.e., the supports of each matrix are pairwise disjoint. 
It was shown in~\cite[Proposition~1]{Dahl} that the number of $S$-permutation matrices is precisely $n!^{2n}$. 
In~\cite[Theorem~1 and Theorem~2]{YordzhevNumber}, an exact formula is given using inclusion-exclusion for the number of \emph{pairs} of disjoint $S$-permutation matrices. 

Permutation matrices were exploited in~\cite{Fontana} for the random sampling of Latin squares of order~$N$, using the same decomposition in~\eqref{Sudoku:decomposition} but using just permutation matrices of order $N$ without the additional block constraint. 
The sampling algorithm in~\cite[Section~3]{Fontana} generates a certain undirected graph whose largest cliques correspond to permutation matrices in~\eqref{Sudoku:decomposition} used to construct a Latin square. 
It is the creation of this undirected graph along with the finding of all cliques which currently dominates the computation. 
For Latin squares of order $n$, the procedure was shown to be an effective method of uniform sampling for $n \leq 7$.

\subsection{Importance sampling}
\label{section:importance}
In~\cite{DeSalvoSudoku}, a backtracking algorithm for Sudoku matrices is presented as follows: 
\begin{itemize}
\item Generate the first row as a random permutation of $\{1,\ldots, 9\}$.
\item Simulate each subsequent row one at a time, and left to right each entry at a time, with $R_{i,j}$ denoting the set of elements in row $i$ and column $j$ which can be placed without violating the Sudoku conditions, by selecting an element from $R_{i,j}$ uniformly at random at each stage. 
\item If for any $2 \leq i \leq 9$ and $1 \leq j \leq 9$ we have $R_{i,j} = \emptyset$, then delete the current row and the preceding row, and continue.
\end{itemize}
This algorithm was used to generate a sample size of $10^8$ in~\cite{DeSalvoSudoku}. 
While no quantitative bounds on the bias were proved, the algorithm completed quickly, always producing a valid Sudoku matrix.

In~\cite{RidderIS}, a similar algorithm is championed, except if the algorithm encounters $R_{i,j} = \emptyset$ for any $i$ and $j$, the algorithm halts. 
The advantage this algorithm has is that one can more easily keep track of the importance sampling weight, and hence establish an approximation for the number of Sudoku matrices as follows. 
If we assume that each row is a uniform permutation of the set $\{1, \ldots, 9\}$, then each entry in column $j$ of a Sudoku matrix can be one of $10-j$ possible values, depending on which of the first $j$ elements were placed first. 
Whereas in the algorithm we select an element uniformly from the set $R_{i,j}$ of size $|R_{i,j}| \leq 10-j$, we correct for this via the likelihood ratio
\[ \ell_{i,j} = \frac{|R_{i,j}|}{10-j}, \qquad R_{i,j} \neq 0,\]
with $\ell_{i,j} = 0$ if $R_{i,j} = \emptyset$. 
The importance sampling weight is then the product of these ratios, i.e., 
\[ \prod_{i=2}^9 \prod_{j=1}^9 \ell_{i,j}. \]
This was used to estimate the approximate number of Sudoku matrices, obtaining $6.662 \cdot 10^{21}$ as a point estimate. 
This algorithm can also be adapted to larger numerical tables, e.g., $(n,n)$-Sudoku matrices, but as the author points out, the probability of generating a feasible matrix, i.e., the probability that the algorithm does not halt before completion, becomes increasingly small.

\subsection{Markov chain approaches}

In~\cite{JacobsonMatthews}, two Markov chains on the set of Latin squares of order~$n$ are presented which involve rather intricate sets of transformations to transition from one state to the next, always yielding a valid Latin square of order~$n$. 
It is proved in~\cite[Theorem~7]{JacobsonMatthews} that the stationary distribution is indeed the uniform distribution, and that transitioning from one state to another requires $O(n)$ arithmetic operations. 
However, as indicated in~\cite[Section~6]{JacobsonMatthews}, there is no proof that the chain is rapidly mixing.  

A Markov chain for $(n,n)$-Sudoku matrices was recently introduced in~\cite{SudokuMC}, via a connection with the set of binary contingency tables (the set of matrices with entries in $\{0,1\}$ with row sums and column sums specified), and then by finding a Markov basis, see~\cite{DiaconisSturmfels}. 
The Markov basis ensures that all states are reachable, however, the technique is admittedly not feasible for the standard $(3,3)$-Sudoku matrices, and instead it is explored using $(2,2)$-Sudoku matrices. 
While an intriguing and unique approach to random sampling, there is also no proof that the chain is rapidly mixing.

\section{Final Remarks}\label{final:remarks}
Our motivation for these algorithms is to verify a claim from \cite{DeSalvoSudoku}, in which the authors calculated and compared the Shannon entropy of a random sample of Sudoku matrices and Latin squares of order $9$ using a backtracking algorithm, see Section~\ref{section:importance}.  
This algorithm is fast, though not necessarily unbiased, and allowed the authors to generate a sample of size $10^8$ that was used in the estimation of Shannon's entropy.
For Sudoku matrices, the entropy was estimated as $1.73312 \pm 0.000173$, and for Latin squares of order $9$ the entropy was estimated as $1.73544\pm 0.0001735$.
Assuming the bias in the backtracking algorithm is relatively small, this calculation suggests that a statistical error of at most $10^{-3}$ is required in order to definitively distinguish a difference in entropies between the set of Sudoku matrices and the set of Latin squares of order $9$, and so we anticipate needing an \emph{unbiased} sample of size at least $10^6$ to effectively distinguish between the entropies of these two sets of matrices.

While Algorithm~\ref{Sudoku} and Algorithm~\ref{Sudoku:n} are provably unbiased, they each took approximately 24 hours to generate a sample of size 1000 on a personal computer\footnote{1.8 GHz Intel Core i7.} (using $n=9$ in Algorithm~\ref{Sudoku:n}).  
Using this i.i.d. sample of size 1000, an unbiased estimate for the entropy of Sudoku matrices is $1.73356 \pm 0.0383708$, and for Latin squares of order 9 our estimate for the entropy is $1.73335 \pm 0.0389771$, which is consistent with the backtracking algorithm.

We have supplied C++ source code used to generate the samples and posted it at \\
\url{https://github.com/stephendesalvo}, along with the files containing the random samples of size 1000, and scripts to load in the matrices into an $n\times n \times 1000$ array in {\tt Matlab} and a $1000 \times n \times n$ array in {\tt Mathematica}.

\section{Acknowledgements}
The author would like to acknowledge helpful conversations with James Zhao and Edo Liberty, and also helpful comments from an anonymous reviewer.

\bibliographystyle{plain}
\bibliography{../../../master_bib}

\end{document}